\documentclass[12pt,reqno]{amsart}

\usepackage{amssymb}
\usepackage{amsthm}
\usepackage{latexsym,amsmath,amssymb}
\usepackage{ascmac} 
\usepackage[dvipdfmx]{graphicx}
\usepackage{braket}
\usepackage{amscd}
\usepackage{amsfonts}
\usepackage{mathrsfs}
\usepackage{bm}  
\usepackage{enumitem}  
\usepackage{hhline} 
\usepackage{appendix}
\usepackage[colorlinks=true, bookmarks=true,
bookmarksnumbered=true, bookmarkstype=toc, linkcolor=blue,
urlcolor=blue, citecolor=blue]{hyperref}
\usepackage{float}
\usepackage{wrapfig}
\restylefloat{figure}
\usepackage{floatflt}

\def\XXint#1#2#3{{\setbox0=\hbox{$#1{#2#3}{\int}$}
\vcenter{\hbox{$#2#3$}}\kern-.5\wd0}}

\theoremstyle{plain} 
\newtheorem{thm}{Theorem}[section]
\newtheorem{lem}[thm]{Lemma}
\newtheorem{cor}[thm]{Corollary}
\newtheorem{prop}[thm]{Proposition}

\theoremstyle{definition} 
\newtheorem{dfn}[thm]{Definition}
\newtheorem{rem}[thm]{Remark}
\makeatletter
    
    \@addtoreset{equation}{section}
  \makeatother

\makeatletter
\def\tr{\mathop{\operator@font tr}\nolimits}
\makeatother

\makeatletter
\def\dist{\mathop{\operator@font dist}\nolimits}
\makeatother

\makeatletter
\def\div{\mathop{\operator@font div}\nolimits}
\makeatother

\makeatletter
\def\exp{\mathop{\operator@font exp}\nolimits}
\makeatother

\makeatletter
\def\essinf{\mathop{\operator@font {\itshape ess}.\inf}\nolimits}
\makeatother

\makeatletter
\def\esssup{\mathop{\operator@font {\itshape ess.}\sup}\nolimits}
\makeatother

\newcommand{\R}{\mathbb{R}}
\newcommand{\N}{\mathbb{N}}
\newcommand{\ZZ}{\mathbb{Z}}
\newcommand{\T}{\mathbb{T}}

\newcommand{\M}{{\mathcal M}}

\def\P{\mathcal{P}}

\def\L{\Lambda}

\def\phi{\varphi}

\def\l{\lambda}

\newcommand{\Tr}{\mathop{\mathrm{Tr}}\nolimits}

\begin{document}

\title[$L_p$ estimates for subsolutions of fully nonlinear equations]{
Second order $L_p$ estimates for subsolutions of fully nonlinear equations
}

\author[H. Dong]{Hongjie Dong}
\address[H. Dong]{Division of Applied Mathematics, Brown University, 182 George Street, Providence, RI 02912, USA}
\email{Hongjie\_Dong@brown.edu}

\author[S. Kitano]{Shuhei Kitano
}
\address[S. Kitano]{
Department of Applied Physics, Waseda University, Tokyo, 169-8555, Japan
}
\email{sk.koryo@moegi.waseda.jp}

\thanks{
H. Dong was partially supported by the NSF under agreement DMS-2055244.
S. Kitano was partially supported by a Waseda University Grant for Special Research Projects (Project number: 2023C-472). 
}

\begin{abstract}
We obtain {new} $L_p$ estimates for subsolutions to fully nonlinear equations. Based on our $L_p$ estimates, we further study several topics such as the third and fourth order derivative estimates for concave fully nonlinear equations, critical exponents of $L_p$ estimates and maximum principles, and the existence and uniqueness of solutions to fully nonlinear equations on the torus with free terms in the $L_p$ spaces or in the space of Radon measures.
\end{abstract}

\maketitle

\section{Introduction}
In this paper, we study the elliptic and parabolic Pucci equations:
\[
-\P^{\pm}(D^2u)=f\quad\mbox{and}\quad v_t-\P^{\pm}(D^2v)=f,
\]
where $\P^{\pm}(M)$ are the Pucci operators defined by
\begin{align*}
    \P^+(M)&:=\Lambda\Tr(M^+)-\lambda\Tr(M^-),\\
    \P^-(M)&:=\lambda\Tr(M^+)-\Lambda\Tr(M^-)
\end{align*}
for $0<\lambda\leq\Lambda$ and a symmetric matrix $M$.

We present two different $W^{2}_p$ estimates for the Pucci equations:
\begin{itemize}
    \item[(1)] If $u$ is a subsolution of $-\P^+(D^2u)=f$ in the unit ball $B_1\subset \R^n$ centered at $0$, then
    \[
    \|D^2u\|_{L_\delta(B_{1/2})}\leq C(\|f\|_{L_n(B_1)}+\|u\|_{L_\infty(B_1)})
    \]
    for some $\delta\in (0,1)$.
    \item[(2)] If $u$ is a subsolution of $-\P^-(D^2u)=f$ in $B_1$ and $\lambda<\Lambda$, then
    \[
    \|D^2u\|_{L_1(B_{1/2})}\leq C(\|f\|_{L_1(B_1)}+\|u\|_{L_1(B_1)}).
    \]
\end{itemize}
We also provide the parabolic versions of these estimates. See Theorems \ref{t2}, \ref{t3}, and \ref{t1}.

\textcolor{black}{
The $W^2_\delta$ estimate was initially established by Lin \cite{L} for $C^{1,1}$ solutions and later extended to viscosity solutions by Caffarelli \cite{C}. 
For more recent studies on the $W^2_\delta$ estimate in the homogeneous case, refer to \cite{ASS, K10,M,NT23,NTar}, which discuss applications of $W^2_\delta$ estimates to $W^2_p$ and $C^{2,\alpha}$ estimates, decay behavior of $\delta$ as $\L/\l \to \infty$, and $W^2_\delta$ estimates for subsolutions. 
In this paper, we show that $W^2_\delta$ estimates also apply to subsolutions of Pucci equations with forcing terms.
Thanks to our relaxed conditions, we further establish $W^4_\delta$ estimates for concave fully nonlinear equations
} 

One notable aspect of the $W^2_1$ estimate (2) is that for the Poisson equation: $-\Delta u=f$, the $W^2_p$ estimate fails at $p=1$, while Calder\'on and Zygmund obtained it for any $p\in (1,\infty)$.
The other is the existence of exponents $p$ where the $W^2_p$ estimate fails, lying between those where the estimate holds.
This is very different from the properties of linear equations expressed in the interpolation theorem.
Caffarelli established the $W^{2}_p$ estimate for fully nonlinear elliptic equations when $p>n$ in \cite{C} and Wang extended it to fully nonlinear parabolic equations when $p>n+1$ in \cite{W1,W2}.
The range of $p$ was further extended by Escauriaza \cite{E} to $p>n-\varepsilon$ in the elliptic case and $p>n+1-\varepsilon$ in the parabolic case, where $\varepsilon=\varepsilon(n,\lambda,\Lambda)>0$ is a small constant.
On the other hand, in the \textcolor{black}{spirit} of Pucci \cite{P}, one can construct examples which demonstrate the failure of the $W^{2}_p$ estimate at $p_k=k\Lambda/(\lambda(k-1)+\Lambda)$ for $k=2,3,...,n$.
In light of these results, one of natural question arises: Does the $W^2_p$ estimate only fails at the Pucci exponents $p_k$?
We consider three possibilities:
\begin{itemize}
    \item[(i)] The $W^2_p$ (or $W^{2,1}_p$) estimates only fail at $p=p_k$,
    \item[(ii)] The $W^2_p$ (or $W^{2,1}_p$) estimates only fail in an interval including $p_k$, and
    \item[(iii)] otherwise.
\end{itemize}
In this paper, we demonstrate that {the} case (ii) occurs for the parabolic Pucci equations when $n\Lambda/(\lambda(n-1)+\Lambda)>n/2+1$. {See Theorem \ref{t4}.}
However it remains unknown whether {the} case (ii) or (iii) is also possible for the elliptic Pucci equations.

The maximum principle is an important tool of analysis to the fully nonlinear equations.
Aleksandrov \cite{A61,A63} established the $L_\infty$ estimates for solutions of the elliptic Pucci equation by the $L_n$ norm of the forcing term.
The range of $p$ for which the maximum principle holds was improved by Fabes and Stroock \cite{FS}.
There are related studies, which concerns the fully nonlinear equations of the form:
\[
-\alpha\Delta u-(1-n\alpha)E_1(D^2u)=f,
\]
where $\alpha\in (0,1/n)$ and $E_1(M)$ is the minimum eigenvalue of $M$.
Pucci \cite{P} conjectured the critical exponent; the maximum principle holds only for $p>\max\{n/2,n(1-(n-1)\alpha\}$ for this equation.
Astala, Iwaniec, and Martin solved {the} Pucci conjecture when $n=2$ in \cite{AIM}.
Kuo and Trudinger obtained the maximum principle when $p>n/2$ is an integer in \cite{KT}.

The parabolic analogue of the maximum principle was obtained by Krylov \cite{K} and later independently by Tso \cite{Tso}, which provides $L_\infty$ estimates by the $L_{n+1}$ norm of the forcing term.

\textcolor{black}{
Another important contribution of this paper is the establishment of global $W^2_p$ estimates and a unique existence theorem for solutions to Pucci equations on the flat torus. The regularity theory of fully nonlinear equations on the flat torus has been primarily motivated by applications to homogenization problems and has mostly been studied in the framework of continuous viscosity solutions (see \cite{Evans}). In this work, we extend the theory to the $L_p$-strong solution framework. It is also worth noting that when the forcing term $f$ belongs to $L_1$ or is a Radon measure, the existence of continuous solutions cannot be expected, rendering the $L_p$-viscosity solution theory inapplicable in this context. Instead, we establish a unique existence theorem for the largest subsolution (see Theorem \ref{thm414}).}

The remainder of this article is organized as follows. In Section 2, we introduce the necessary notation and definitions. Section 3 focuses on establishing the $W^{2,1}_\delta$ estimates for subsolutions of the Pucci equations. In Section 4, we address the $W^{2,1}_1$ estimates and examine the properties of critical exponents related to $W^{2,1}_p$ estimates and maximum principles. Finally, in Section 5, we derive global estimates and establish existence and uniqueness results for the Pucci equations on the torus.

\section{Notation}
We list some notation:
\begin{itemize}
    \item $B_r$ is the open ball of radius $r>0$ centered $0$,
    \item $Q_r=B_r\times(-r^2,0)$ is the open cylinder,
    \item $\partial B_r$ is the boundary of $B_r$,
    \item $\partial_p (B_r\times (a,b)):=\partial B_r\times (a,b)\cup \overline{B_r}\times \{a\}$ is the parabolic boundary of $B_r\times (a,b)$,
    \item $\|f\|_{L_p(\Omega)}$ is the $L_p(\Omega)$ norm of $f$,
    \item $Du$ is the gradient $(u_{x_i})$ and $D^2u$ is the spatial Hessian matrix $(u_{x_i,x_j})$,
    \item $X^+,X^-\geq O$ are the positive and negative parts of a symmetric matrix $X$,
    \item $|X|$ is the sum of the absolute value of the eigenvalues of $X$,
    \item $x\otimes y$ is the tensor of $x,y\in \R^n$,
    \item $W^2_p(\Omega)$ is the Sobolev space; $u\in W^2_p(\Omega)$ iff $u,Du,D^2u\in L_p(\Omega)$,
    \item $W^{2,1}_p(\Omega\times (a,b))$ is the parabolic Sobolev space; $v\in W^{2,1}_p(\Omega\times (a,b))$ iff $v$, $Dv$, $D^2v$, $v_t\in L_p(\Omega\times (a,b))$.
    \item $\M(\Omega)$ is the space of Radon measures.
    We define
    \[
    \|f\|_{\M(\Omega)}=\sup\left\{\int_{\Omega}\phi df:\phi\in C_0(\Omega),\|\phi\|_{L_\infty(\Omega)}\leq1\right\}.
    \]
\end{itemize}


We recall the definition of $L_p$-viscosity solutions in \cite{CCKS}.
\begin{dfn}\label{d22}
    Let $\Omega\subset\R^n$ be open and $f\in L_p(\Omega)$.
    A function $u\in C(\Omega)$ is an $L_p$-viscosity subsolution of $-\P^{\pm}(D^2u)=f$ in $\Omega$ if for all $\phi\in W^2_{p,\text{loc}}(\Omega)$, whenever $\theta>0$ , $\mathcal{O}\subset\Omega$ is open and
    \[
    -\P^{\pm}(D^2\phi(x))-f(x)\geq \theta\quad\mbox{a.e. in }\mathcal{O},
    \]
    then $u-\phi$ cannot have a local maximum in $\mathcal{O}$.

    A function $u\in C(\Omega)$ is an $L_p$-viscosity supersolution of $-\P^{\pm}(D^2u)=f$ in $\Omega$ if $v:=-u$ is an $L_p$-viscosity subsolution of $-\P^{\pm}(-D^2v)=f$ in $\Omega$.

    Finally, $u\in C(\Omega)$ is an $L_p$-viscosity solution of $-\P^{\pm}(D^2u)=f$ in $\Omega$ if is both an $L_p$-viscosity subsolution and an $L_p$-viscosity supersolution of $-\P^{\pm}(D^2u)=f$ in $\Omega$.
\end{dfn}
We have analogous definitions of $L_p$-viscosity solutions of the elliptic equations: $\tau u-\P^{\pm}(D^2u)=f$ and the parabolic equations: $u_t-\P^{\pm}(D^2u)=f$.

\section{\texorpdfstring{$W^{2,1}_\delta$}{} estimates for {Pucci maximum} equations}
\subsection{Elliptic Case}
We first prove the $W^2_\delta$ estimate for the elliptic Pucci equations.
\begin{thm}\label{t2}
    There exist constants $C_0=C_0(n,\lambda,\Lambda)>0$ and $\delta_0=\delta_0(n,\lambda,\Lambda)\in (0,1)$ such that for $f\in L_{n}(B_1)$, if $u\in W^{2}_n(B_1)$ satisfies
\begin{equation*}
 -\P^+(D^2u)\leq f\quad\mbox{a.e. in }B_1,
\end{equation*}
then
\begin{equation*}
    \|D^2u\|_{L_{\delta_0}(B_{1/2})}\leq C_0(\|f\|_{L_n(B_1)}+\|u\|_{L_\infty(B_1)}).
\end{equation*}
\end{thm}
For $h>0$, let $\overline{G}_h(u,B_{1/2})\subset B_{1/2}$ be such that $x_0\in \overline{G}_h(u,B_{1/2})$ if and only if there exists a convex paraboloid $P$ of opening $h>0$ such that $P(x_0)=u(x_0)$ and $u(x)\leq P(x)$ for $x\in B_{1/2}$.
We also define $\overline{A}_h(u,B_{1/2}):=B_{1/2}\setminus \overline{G}_h(u,B_{1/2})$.
We use the following property obtained by Caffarelli \cite{C} (see Corollary 3 in \cite{C}):
\begin{prop}\label{prop11}
    There exist constants $\mu_0,C_1,m_0>0$ with the following property:
    Let $u,f$ be as in Theorem \ref{t2}.
    Suppose $\|u\|_{L_\infty(B_1)}\leq 1$ and $\|f\|_{L_n(B_{1})}\leq m_0$.
    Then
    \begin{equation*}
        |\overline{A}_h(u,B_{1/2})|\leq C_1 h^{-\mu_0}
    \end{equation*}
    for $h>0$.
\end{prop}

\begin{proof}[Proof of Theorem \ref{t2}]
    Without loss of generality, we may suppose $\|u\|_{L_\infty({B_1)}}\leq 1$ and $\|f\|_{L_n({B_1)}}\leq m_0$.
    Note that
    \begin{equation*}
        \overline{G}_h(u,B_{1/2})\subset B_{1/2}\cap\{D^2u^+\leq h\mathrm{I}\}.
    \end{equation*}
    Therefore, combining {with} Proposition \ref{prop11}, we have
    \begin{equation*}
        |B_{1/2}\cap\{D^2u^+> hI\}|\leq |\overline{A}_h(u,B_{1/2})|\leq C_1 h^{-\mu_0}.
    \end{equation*}
    This implies
    \begin{equation*}
        \|D^2u^+\|_{L_{\delta_0}(B_{1/2})}\leq C
    \end{equation*}
     for $0<\delta_0<\mu_0$.
     On the other hand, we have
     \begin{equation*}
         -\P^+(D^2u)=-\Lambda\Tr(D^2u^+)+\lambda\Tr(D^2u^-)\leq f
     \end{equation*}
     and thus
     \begin{equation*}
         \lambda\Tr(D^2u^-)\leq \Lambda\Tr(D^2u^+)+f\quad\mbox{in }B_{1/2}.
     \end{equation*}
     This is the key observation in order to relax the condition.
     Hence
     \begin{equation*}
         \|D^2u^-\|_{L_{\delta_0}(B_{1/2})}\leq C(\|D^2u^+\|_{L_{\delta_0}(B_{1/2})}+\|f\|_{L_{\delta_0}(B_{1/2})}).
     \end{equation*}
     The theorem is proved.
\end{proof}
\begin{cor}\label{c2}
Let $F$ be concave and elliptic in the sense that
\[
\P^-(X-Y)\leq F(X)-F(Y)\leq \P^+(X-Y)
\]
for symmetric matrices $X,Y$.
For $f\in W^2_{n}(B_1)$, if $u\in W^{4}_n(B_1)$ is a solution of
\begin{equation*}
    -F(D^2u)=f\quad\mbox{a.e. in }B_1,
\end{equation*}
then
\begin{equation}\label{e331}
    \|D^4u\|_{L_{\delta_0}(B_{1/2})}\leq C_0(\|D^2f\|_{L_n(B_1)}+\|D^2u\|_{L_\infty(B_1)}),
\end{equation}
where $C_0$ and $\delta_0$ are as in Theorem \ref{t2}.
\end{cor}
\begin{proof}
    For $h>0$ and $e\in\partial B_1$, using the concavity of $F$, we have
    \[
    \textcolor{black}{-F\left(D^2\left(\frac{u(x+he)+u(x-he)}{2}\right)\right)\leq -\frac{1}{2}(F(D^2u(x+he))-F(D^2u(x-he)))}
    \]
    and from the assumption,
    \textcolor{black}{
    \begin{align*}
    &-\P^+\left(D^2\left(\frac{u(x+he)+u(x-he)-2u(x)}{2}\right)\right)\\
    &\leq -F\left(D^2\left(\frac{u(x+he)+u(x-he)}{2}\right)\right)+F(D^2u(x)) \\
    &\leq \frac{1}{2}(f(x+he)+f(x-he))-f(x).
    \end{align*}
    }
    Hence by letting $h\to0$,
    \begin{equation*}
        -\P^+(D^2u_{ee})\leq f_{ee}\quad\mbox{in }B_1.
    \end{equation*}
    Therefore \eqref{e331} follows from Theorem \ref{t2}.
\end{proof}
\subsection{Parabolic Case} Next we present the parabolic analogue of the estimate.
\begin{thm}\label{t3}
    There exist constants $C_2=C_2(n,\lambda,\Lambda)>0$ and $\delta_1=\delta_1(n,\lambda,\Lambda)\in (0,1)$ such that for $f\in L_{n+1}(Q_{1})$, if $v\in W^{2,1}_{n+1}(Q_{1})$ satisfies
    \[
    v_t-\P^+(D^2v)\leq f\quad\mbox{a.e. in }Q_{1},
    \]
    then
    \begin{align}
    &\|v_t\|_{L_{\delta_1}(Q_{1/2})}+\|D^2v\|_{L_{\delta_1}(Q_{1/2})}\leq C_2(\|f\|_{L_{n+1}(Q_{1})}+\|v\|_{L_\infty(Q_{1})}).\label{e2}
\end{align}
\end{thm}
\begin{proof}
    We basically use the same idea {as in the proof of} Theorem \ref{t2}.
    The parabolic version of Proposition \ref{prop11} {was} obtained by Wang \cite[Theorem 4.11]{W1}. {See also \cite{K10}.}
    Especially, there exist $m_1,\mu_1\in (0,1)$ and $C_3>0$ such that assuming without loss of generality that $\|u\|_{L_\infty(Q_{1})}\leq 1$ and $\|f\|_{L_{n+1}(Q_{1})}\leq m_1$,
    \[
    \left|Q_{1/2}\cap  (\{D^2v^+>hI\}\cup\{v_t<-h\})\right|\leq C_3h^{-\mu_1}
    \]
    for $h>0$.
    Hence
    \[
    \|v^-_t\|_{L_\delta(Q_{1/2})}+\|D^2v^+\|_{L_\delta(Q_{1/2})}\leq C.
    \]
    We also have
    \[
    \|v^+_t\|_{L_\delta(Q_{1/2})}+\|D^2v^-\|_{L_\delta(Q_{1/2})}\leq C
    \]
    by using
    \[
    v_t^++\lambda\Tr(D^2v^-)\leq \Lambda\Tr(D^2v^+)+v^-_t+f\quad\mbox{in }Q_{1/2}.
    \]
    The theorem is proved.
\end{proof}

\begin{cor}
Let $C_2$ and $\delta_1$ be as in Theorem \ref{t3} and $F$ be concave and elliptic as in Corollary \ref{c2}.
For $f\in W^{2,1}_{n+1}(Q_{1})$, if $v\in W^{4,3}_{n+1}(Q_{1})$ is a solution of
\begin{equation*}
   v_t- F(D^2v)=f\quad\mbox{a.e. in }Q_{1},
\end{equation*}
then
\begin{align*}
    &\|v_{tee}\|_{L_{\delta_1}(Q_{1/2})}+\|D^2v_{ee}\|_{L_{\delta_1}(Q_{1/2})}\leq C_2(\|f_{ee}\|_{L_{n+1}(Q_{1})}+\|u_{ee}\|_{L_\infty(Q_{1})})
\end{align*}
for $e\in \partial B_1$ and
\textcolor{black}{
\begin{align*}
    \|v_{ttt}\|_{L_{\delta_1}(Q_{1/2})}+\|D^2v_{tt}\|_{L_{\delta_1}(Q_{1/2})}
    \leq C_2(\|f_{tt}\|_{L_{n+1}(Q_{1})}+\|u_{tt}\|_{L_\infty(Q_{1})}).
\end{align*}
}
\end{cor}
\begin{proof}
By a similar way with Corollary \ref{c2}, we have
\[
v_{tee}-\P^{+}(D^2v_{ee})\leq f_{ee}\quad\mbox{in }Q_{1}
\]
for $e\in \partial B_1$ and
\textcolor{black}{
\[
v_{ttt}-\P^{+}(D^2v_{tt})\leq f_{tt}\quad\mbox{in }Q_{1}.
\]
}
We obtain the assertion by applying Theorem \ref{t3}.
\end{proof}

\section{\texorpdfstring{$W^{2,1}_1$}{1} estimates and Structure of Critical Exponents}
\subsection{\texorpdfstring{$W^{2,1}_1$}{1} estimates} Here we provide the $W^{2,1}_1$ estimate for subsolutions of the parabolic Pucci equations.
\textcolor{black}{In what follows, we denote $K:=B_{1/2}\times (-2/3,-1/3)$.}
\begin{thm}\label{t1}
Suppose that $\lambda<\Lambda$.
There exists $C_3=C_3(n,\lambda,\Lambda)>0$ such that for $f\in L_1(Q_{1})$, if $v\in W^{2,1}_1(Q_{1})$ satisfies
\begin{equation*}
 v_t-\P^-(D^2 v)\leq f\quad\mbox{a.e. in }Q_{1},
\end{equation*}
then
\begin{equation}\label{e1}
    \|v_t\|_{L_1(K)}+\|D^2v\|_{L_1(K)}\leq C_3(\|f\|_{L_1(Q_{1})}+\|v\|_{L_1(Q_{1})}).
\end{equation}
\end{thm}
The $W^2_1$ estimate for the elliptic Pucci equations follows from Theorem \ref{t1}.
\begin{cor}\label{c1}
Suppose that $\lambda<\Lambda$.
There exists $C_4=C_4(n,\lambda,\Lambda)>0$ such that \textcolor{black}{for $f\in L_1(B_1)$,} if $u\in W^{2}_{1}(B_1)$ satisfies
\begin{equation*}
 -\P^-(D^2u)\leq f\quad\mbox{a.e. in }B_1,
\end{equation*}
then
\begin{equation}\label{e3}
    \|D^2u\|_{L_1(B_{1/2})}\leq C_4(\|f\|_{L_1(B_1)}+\|u\|_{L_1(B_1)}).
\end{equation}
\end{cor}
\begin{proof}
    This immediately follows by applying Theorem \ref{t1} to $v(x,t):=u(x)$ in $Q_{1}$.
\end{proof}

\begin{rem}
In Theorem \ref{t1} and Corollary \ref{c1}, the assumption $\lambda<\Lambda$ is necessary.
We provide counterexamples of the estimates \eqref{e1} and \eqref{e3} for $\lambda=\Lambda$ in Proposition \ref{P1}.
\end{rem}
\textcolor{black}{
\begin{rem}
    One may wonder what if $f\in C(Q_1)$ and $v\in C(Q_1)$ is only a viscosity subsolution in Theorem \ref{t1}.
    In this case, from the convexity of the equation, it follows that
    \[
    v_t^\delta-\P^-(D^2v^\delta)\leq f^\delta\quad Q_{4/5},
    \]
    where $v^\delta=v*\eta^\delta$ is a mollification of $v$, with a mollifier $\eta$ in $\R^{n+1}$.
    We may apply Theorem \ref{t1} to $v^\delta$.
    By sending $\delta\to 0$, we obtain the Hessian estimate in the space of Radon measures $\M$:
    \[
    \|v_t\|_{\M(K)}+\|D^2v\|_{\M(K)}\leq C_3(\|f\|_{L_1(Q_{1})}+\|v\|_{L_1(Q_{1})}).
    \]
\end{rem}
}

\begin{lem}\label{l1}
There is a constant $C_5=C_5(n,\lambda,\Lambda)>0$ such that if $v\in W^{2,1}_1(Q_{1})$ satisfies
\begin{equation*}
    v_t-\frac{\Lambda+\lambda}{2}\Delta v\leq f\quad\mbox{a.e. in }Q_{1},
\end{equation*}
then
\begin{equation}\label{e4}
   \left\|v_t-\frac{\Lambda+\lambda}{2}\Delta v\right\|_{L_1(K)} \leq C_5(\|f\|_{L_1(Q_{1})}+\|u\|_{L_1(Q_{1})}).
\end{equation}
\end{lem}

\begin{proof}
Let $\eta$ be a cutoff function such that $\eta\equiv1$ in $K$ and is compactly supported in $Q_{1}$.
For any $g\in L_\infty(Q_{1})$ with support contained in $K$, it follows that
\begin{equation*}
    \int_{Q_{1}}\left(v_t-\frac{\Lambda+\lambda}{2}\Delta v-f\right)(g+\|g\|_\infty\eta)dxdt\leq0
\end{equation*}
since $v_t-(\Lambda+\lambda)\Delta v/2-f\leq 0$ and $g+\|g\|_\infty\eta\geq0$ in $Q_1$.
\textcolor{black}{
Therefore, using the integration by parts, it follows that
\begin{align*}
    &\quad\int_{Q_{1}}\left(v_t-\frac{\Lambda+\lambda}{2}\Delta v\right)gdxdt\\
    &\leq \int_{Q_{1}}\left(v\|g\|_\infty\left(\eta_t+\frac{\Lambda+\lambda}{2}\Delta \eta\right)+f\left(g+\|g\|_\infty\eta\right)\right)dxdt.
\end{align*}
}
Hence we have
\begin{equation*}
    \int_{Q_{1}}\left(v_t-\frac{\Lambda+\lambda}{2}\Delta v\right)g\leq C\|g\|_{L_\infty(Q_{1})} (\|f\|_{L_1(Q_{1})}+\|v\|_{L_1(Q_{1})}).
\end{equation*}
By replacing $g$ with $-g$, we also have
\begin{equation*}
\int_{Q_{1}}\left(v_t-\frac{\Lambda+\lambda}{2}\Delta v\right)g\geq -C\|g\|_{L_\infty(Q_{1})} (\|f\|_{L_1(Q_{1})}+\|v\|_{L_1(Q_{1})}).
\end{equation*}
Hence \eqref{e4} follows.
\end{proof}

\begin{proof}[Proof of Theorem \ref{t1}]
Using
\begin{align}
                \label{eq3.44}
    \P^-(D^2v)=\frac{\Lambda+\lambda}{2}\Delta v-\frac{\Lambda-\lambda}{2}\Tr(D^2v^++D^2v^-)\leq \frac{\Lambda+\lambda}{2}\Delta v,
\end{align}
we obtain
\begin{equation*}
    f\geq v_t-\P^-(D^2v)\geq v_t-\frac{\Lambda+\lambda}{2}\Delta v\quad\mbox{a.e. in }Q_{1}.
\end{equation*}
We compute
\begin{align*}
    \frac{\Lambda-\lambda}{2}|D^2v|
    &=\frac{\Lambda-\lambda}{2}\Tr(D^2v^++D^2v^-)\\
    &=\frac{\Lambda+\lambda}{2}\Delta v-\P^-(D^2v)
    \leq f-v_t+\frac{\Lambda+\lambda}{2}\Delta v.
\end{align*}
Hence by Lemma \ref{l1}
\begin{align*}
    \|D^2v\|_{L_1(K)}
    &\leq \frac{C}{\Lambda-\lambda}\left(\|f\|_{L_1(K)}+\left\|v_t-\frac{\Lambda+\lambda}{2}\Delta v\right\|_{L_1(K)}\right)\\
    &\leq \frac{C}{\Lambda-\lambda}(\|f\|_{L_1(Q_{1})}+\|v\|_{L_1(Q_{1})}).
\end{align*}
We also obtain
\begin{align*}
    \|v_t\|_{L_1(K)}
    &\leq C\left(\|\Delta v\|_{L_1(K)}+\left\|v_t-\frac{\Lambda+\lambda}{2}\Delta u\right\|_{L_1(K)}\right)\\
    &\leq \frac{C}{\Lambda-\lambda}(\|f\|_{L_1(Q_{1})}+\|v\|_{L_1(Q_{1})}).
\end{align*}
The theorem is proved.
\end{proof}

\subsection{Some Properties of Critical Exponents}
We next study the critical exponents of the $W^2_p$ $(W^{2,1}_p)$ estimates and the generalized maximum principle.
The next proposition is a modification of Pucci \cite{P}.
\begin{prop}\label{P1}
    For $p_n=n\Lambda/(\lambda (n-1)+\Lambda)$, it follows that
    \begin{align}\label{p1}
        \sup_{U\in W^{2}_{p_n}(B_1)}\left\{\frac{\|U\|_{L_\infty(B_1)}-\|U\|_{L_\infty(\partial B_1)}}{\|\P^-(D^2U)\|_{L_{p_n}(B_1)}}\right\}&=\infty.
    \end{align}
    For $k=2,3,...,n$ and $p_k=k\Lambda/(\lambda (k-1)+\Lambda)$, the following properties hold:
    \begin{align}
        \sup_{U\in W^{2}_{p_k}(B_1)}\left\{\frac{\|D^2U\|_{L_{p_k}(B_{1/2})}}{\|\P^-(D^2U)\|_{L_{p_k}(B_1)}+\|U\|_{L_{p_k}(B_1)}}\right\}&=\infty\label{p2}\quad \mbox{and }\\
        \sup_{V\in W^{2,1}_{p_k}(B_1\times(0,1) )}\left\{\frac{\|V_t\|_{L_{p_k}(Q_{1/2})}+\|D^2V\|_{L_{p_k}(Q_{1/2})}}{\|(V_t-\P^-(D^2V))\|_{L_{p_k}(Q_{1})}+\|V\|_{L_{p_k}(Q_{1} )}}\right\}&=\infty.\label{p3}
    \end{align}
\end{prop}
\begin{proof}
    We first prove the assertion in the case $k=n$.
    For $\theta\in\R$ such that
    \begin{equation}\label{p6}
    {2-n}\leq 1-\lambda(n-1)/\Lambda<\theta<1,
    \end{equation}
    set
    \[
    U_\theta(x):=
    \left\{
        \begin{array}{cc}
        |x|^{\theta}-1     & \text{if }0\leq\theta<1\quad\text{and} \\
        1-|x|^\theta   & \text{if }\theta<0.
        \end{array}
    \right.
    \]
    In order to prove \eqref{p1} and \eqref{p2}, it is enough to see
    \begin{align}
        \frac{\|U_\theta\|_{L_\infty(B_1)}-\|U_\theta\|_{L_\infty(\partial B_1)}}{\|\P^-(D^2U_\theta)\|_{L_{p_n}(B_1)}}&\to\infty\quad\mbox{and}\label{p4}\\
        \frac{\|D^2U_\theta\|_{L_{p_n}(B_{1/2})}}{\|\P^-(D^2U_\theta)\|_{L_{p_n}(B_1)}+\|U_\theta\|_{L_{p_n}(B_1)}}&\to\infty\quad\mbox{as }\theta\to 1-\lambda(n-1)/\Lambda.\label{p5}
    \end{align}
    Noting $D^2U_\theta\in L_{p_n}(B_1)$ from \eqref{p6},
    we compute
    \[
    D^2U_\theta(x)=
    \left\{
    \begin{array}{cc}
    \frac{\theta}{|x|^{2-\theta}}\left((\theta-1)\frac{x}{|x|}\otimes\frac{x}{|x|}+\left(I-\frac{x}{|x|}\otimes\frac{x}{|x|}\right)\right)&\text{if }0\leq\theta<1,\\
    -\frac{\theta}{|x|^{2-\theta}}\left((\theta-1)\frac{x}{|x|}\otimes\frac{x}{|x|}+\left(I-\frac{x}{|x|}\otimes\frac{x}{|x|}\right)\right)&\text{if }{2-n<}\theta< 0.
    \end{array}
    \right.
    \]
    Therefore, for all $\theta<1$, it follows that
    \[
    D^2U_\theta(x)^+=\frac{|\theta|}{|x|^{2-\theta}}\left(I-\frac{x}{|x|}\otimes\frac{x}{|x|}\right),\quad D^2U_\theta(x)^-=\frac{|\theta|(1-\theta)}{|x|^{2-\theta}}\left(\frac{x}{|x|}\otimes\frac{x}{|x|}\right)
    \]
    and
    \[
    \P^-(D^2U_\theta)=
    |\theta|(\lambda(n-1)-\Lambda(1-\theta))|x|^{-2+\theta}.
    \]
    Hence it follows that
    \begin{align*}
        \|\P^-(D^2U_\theta)\|_{L_{p_n}(B_1)}
        &=|\theta|(\lambda(n-1)-\Lambda(1-\theta))\left(|\partial B_1|\int_{0}^1r^{n-1-{p_n}(2-\theta)}dr\right)^{1/{p_n}}\\
        &=|\theta|(\lambda(n-1)-\Lambda(1-\theta))\left(\frac{|\partial B_1|}{n-{p_n}(2-\theta)}\right)^{1/{p_n}}\\
        &=|\theta|(\lambda(n-1)-\Lambda(1-\theta))^{1-1/{p_n}}\left(\frac{|\partial B_1|(\lambda(n-1)+\Lambda)}{n}\right)^{1/{p_n}},
    \end{align*}
{    where we used $n-p_n(2-\theta)>0$.}
    We also obtain
    \[
    \|U_\theta\|_{L_\infty(B_1)}-\|U_\theta\|_{L_\infty(\partial B_1)}
    =\left\{
    \begin{array}{cc}
        1 &  \text{if }0<\theta<1,\\
        \infty & \text{if }{2-n<}\theta<0
    \end{array}
    \right.
    \]
    Hence if $\theta<0$, \eqref{p4} immediately follows.
    If $\theta>0$,
    \[
    \frac{\|U_\theta\|_{L_\infty(B_1)}-\|U_\theta\|_{L_\infty(\partial B_1)}}{\|\P^-(D^2U_\theta)\|_{L_{p_n}(B_1)}}\asymp(\lambda(n-1)-\Lambda(1-\theta))^{-1+1/{p_n}}\to\infty\quad\mbox{as }\theta\to 1-\frac{\lambda(n-1)}{\Lambda}.
    \]
    To prove \eqref{p5}, we compute
    \begin{align*}
        \|U_\theta\|_{L_{p_n}(B_1)}
        &\textcolor{black}{=\left(|\partial B_1|\int_0^1r^{n-1}|1-r^{\theta}|^{p_n}dr\right)^{1/{p_n}}}\\
        &=\left(|\partial B_1|\int_0^1r^{n-1-{p_n}\theta^{-}}(1-r^{|\theta|})^{p_n}dr\right)^{1/{p_n}}\\
        &=\left(\frac{|\partial B_1|}{|\theta|}\int_0^1s^{(n-|\theta|-{p_n}\theta^{-})/|\theta|}(1-s)^{p_n}ds\right)^{1/{p_n}}\\
        &=\left(\frac{|\partial B_1|}{|\theta|}B\left(\frac{n-p_n\theta^-}{|\theta|},{p_n}+1\right)\right)^{1/{p_n}}
    \end{align*}
    and
    \begin{align*}
        \|D^2U_\theta\|_{L_{p_n}(B_{1/2})}
        &=|\theta|(n-\theta)\left(|\partial B_1|\int_{0}^{1/2}r^{n-1-{p_n}(2-\theta)}\right)^{1/{p_n}}\\
        &=|\theta|(n-\theta)\left(\frac{|\partial B_1|(\lambda(n-1)+\Lambda)}{n2^{n-p_n(2-\theta)}(\lambda(n-1)-\Lambda(1-\theta))}\right)^{1/{p_n}},
    \end{align*}
    where $B(x,y)$ is the beta function for $x,y>0$.
    In the case that $\theta\to 1-\lambda(n-1)/\Lambda\neq0$, we have
    \[
    \frac{\|D^2U_\theta\|_{L_{p_n}(B_{1/2})}}{\|\P^-(D^2U_\theta)\|_{L_{p_n}(B_1)}+\|U_\theta\|_{L_{p_n}(B_1)}}
    \asymp(\lambda(n-1)-\Lambda(1-\theta))^{-1/{p_n}}
    \to\infty.
    \]
    In the case that $1-\lambda(n-1)/\Lambda=0$, noting $B(x,y)\asymp x^{-y}$ as $x\to\infty$, we have
    \[
    \frac{\|D^2U_\theta\|_{L_{p_n}(B_{1/2})}}{\|\P^-(D^2U_\theta)\|_{L_{p_n}(B_1)}+\|U_\theta\|_{L_{p_n}(B_1)}}
    \asymp\theta^{-1/{p_n}}
    \to\infty\quad\mbox{as }\theta\to0.
    \]
    Finally, \eqref{p3} immediately follows since for $V_\theta(x,t):=U_\theta(x)$, we have
    \begin{align*}
    &\frac{\|(V_\theta)_t\|_{L_{p_n}(Q_{1/2})}+\|D^2V_\theta\|_{L_{p_n}(Q_{1/2})}}{\|(V_\theta)_t-\P^-(D^2V_\theta)\|_{L_{p_n}(Q_{1})}+\|V_\theta\|_{L_{p_n}(Q_{1})}}\\
    &=\frac{C\|D^2U_\theta\|_{L_{p_n}(B_{1/2})}}{\|\P^-(D^2U_\theta)\|_{L_{p_n}(B_1)}+\|U_\theta\|_{L_{p_n}(B_1)}}
 \to\infty.
    \end{align*}

    When $k=2,...,n-1$, it is enough to construct a function in the same way but depending only on the first $k$ variables as follows:
    \[
    U_\theta^k(x)
    =\left\{
        \begin{array}{cl}
        |\tilde{x}|^{\theta}-1     & \text{if }0\leq\theta<1\quad\text{and} \\
        1-|\tilde{x}|^\theta   & \text{if }\theta<0
        \end{array}
    \right.
    \]
    for $x=(x_1,x_2,...,x_n)\in \R^n$ and $\tilde{x}=(x_1,x_2,...,x_k)\in \R^k$.
\end{proof}
\begin{rem}
    Note that $U_\theta$ is a non-trivial solution of the Pucci equation:
    \begin{equation*}
        \left\{
        \begin{split}
            -\P^-(D^2U_\theta)&=0\quad\mbox{in }B_1\setminus\{0\},\\
            U_\theta&=0\quad\mbox{in }\partial B_1
        \end{split}
        \right.
    \end{equation*}
    for $\theta=1-\lambda(n-1)/\Lambda$.
\end{rem}
In view of Proposition \ref{P1}, one natural question arises: Is it only the Pucci exponents $p_k$ where the $W^{2,1}_p$ estimates fails?
Here we answer the question and show that the $W^{2,1}_p$ estimate fails not only the Pucci exponent but also in an interval when $\Lambda/\lambda$ is large enough, in the parabolic setting.
We denote by $p_M\in [n/2+1,\infty)$ the critical exponent of generalized maximum principles for fully nonlinear parabolic equations:
\[
p_M:=\inf\left\{p> \frac{n}{2}+1:\sup_{v\in W^{2,1}_p}\frac{\|v\|_{L_\infty(Q_{1} )}-\|v\|_{L_\infty(\partial_pQ_{1})}}{\|v_t-\P^-(D^2v)\|_{L_p(Q_{1})}}<\infty\right\}.
\]
We also denote by $p_W\in [n/2+1,\infty)$ the critical exponent of $W^{2,1}_p$ estimates:
\[
p_W:=\inf\left\{p> \frac{n}{2}+1:\sup_{v\in W^{2,1}_p}\frac{\|v_t\|_{L_p(Q_{1/2})}+\|D^2v\|_{L_p(Q_{1/2})}}{\|v_t-\P^-(D^2v)\|_{L_p(Q_{1} )}+\|v\|_{L_p(Q_{1})}}<\infty\right\}.
\]
Note that $p_M,p_W\leq n+1-\varepsilon_1$ for some $\varepsilon_1=\varepsilon_1(n,\lambda,\Lambda)>0$ from \cite{CFKS,E,FS}.
\begin{thm}\label{t4}
It follows that $p_M=p_W\geq \max\{p_n,n/2+1\}$.
Especially, if
\[
\frac{\Lambda}{\lambda}>\frac{(n+2)(n-1)}{n-2}
\]
(so that $p_n=n\Lambda/((n-1)\lambda+\Lambda)>n/2+1$), then it follows that
\begin{equation}\label{e471}
    \sup_{v\in W^{2,1}_p}\frac{\|v_t\|_{L_p(Q_{1/2})}+\|D^2v\|_{L_p(Q_{1/2})}}{\|v_t-\P^-(D^2v)\|_{L_p(Q_{1} )}+\|v\|_{L_p(Q_{1})}}=\infty
\end{equation}
for any $n/2+1<p\leq p_n$.
\end{thm}
\begin{proof}
    It was proved by Escauriaza \cite{E} (see also Theorem 1.4 in \cite{CFKS}) that for $p>n/2+1$, the generalized maximum principle:
    \[
    \sup_{v\in W^{2,1}_p}\frac{\|v\|_{L_\infty(Q_{1})}-\|v\|_{L_\infty(\partial_pQ_{1/2})}}{\|v_t-\P^-(D^2v)\|_{L_p(Q_{1})}}<\infty
    \]
    implies that the $W^{2,1}_q$ estimate holds for any $q>p$:
    \[
    \sup_{v\in W^{2,1}_q}\frac{\|v_t\|_{L_q(Q_{1/2})}+\|D^2v\|_{L_q(Q_{1/2})}}{\|v_t-\P^-(D^2v)\|_{L_q(Q_{1} )}+\|v\|_{L_q(Q_{1} )}}<\infty.
    \]
    Hence we have $p_W\leq p_M$.

    We also obtain $p_M\geq p_n$ by the same way.
    Indeed, if $p_M<p_n$ holds, this implies the $W^{2,1}_{p_n}$ estimate as {explained} above, which contradicts \eqref{p3}.

    It remains to prove $p_M\leq p_W${, which together with $p_M\ge p_n$ implies that the $W^{2,1}_p$ estimate cannot hold for $p\in (n/2+1,p_n]$, i.e., \eqref{e471}.} 
    To this end, we aim to show that for any $p>n/2+1$, if there exists $C_6>0$ such that the $W^{2,1}_p$ estimate:
    \begin{align}
    \|v_t\|_{L_p(Q_{1/2})}+\|D^2v\|_{L_p(Q_{1/2})}\leq C_6(\|v_t-\P^-(D^2v)\|_{L_p(Q_{1})}+\|v\|_{L_p(Q_{1})})\label{e162}
    \end{align}
    holds for $v\in W^{2,1}_p(Q_{1})$, then there is also a constant $C_7>0$ such that for any $v\in C^{2}(\overline{Q_{1}})$,
    \[
    \|v\|_{L_\infty(Q_{1} )}-\|v\|_{L_\infty(\partial_pQ_{1})}\leq C_7\|v_t-\P^-(D^2v)\|_{L_p(Q_{1})}.
    \]
    Since $C^{2}(\overline{Q_{1}})$ is dense in $W^{2,1}_p(Q_{1})$, this implies
    \[
    \sup_{v\in W^{2,1}_p}\frac{\|v\|_{L_\infty(Q_{1} )}-\|v\|_{L_\infty(\partial_pQ_{1})}}{\|v_t-\P^-(D^2v)\|_{L_p(Q_{1})}}\leq C_7,
    \]
    which completes the proof.

    Fix $v\in C^{\infty}(\overline{Q_{1}})$ and let $f\in C(\overline{Q_{1}})$ be such that
    \[
    f:=v_t-\P^-(D^2v)\quad\mbox{in }Q_{1}.
    \]
    Then we have
    \[
    v_t-\Delta v\leq v_t-\P^-(D^2v)=f\quad\mbox{in }Q_{1},
    \]
    and applying the generalized maximum principle,
    \[
    \sup_{Q_{1}}v-\sup_{\partial_pQ_{1}}v\leq C\|f\|_{L_p(Q_{1} )}.
    \]
    Hence it remains to prove
    \begin{equation}\label{e161}
    -\inf_{Q_{1}}v+\inf_{\partial_pQ_{1}}v\leq C\|f\|_{L_p(Q_{1})}.
    \end{equation}
    To obtain \eqref{e161}, we will show $v-\inf_{\partial_pQ_{1}}v$ is bounded from below by $w\in W^{2,1}_{q, \text{loc}}(\R^n\times [-1,0))$ for $q> n+1-\varepsilon_1$, which is a unique \textcolor{black}{strong} solution of
    \[
    \left\{
    \begin{array}{ll}
    w_t-\P^-(D^2w)=-f^-\chi_{Q_{1}}     &\mbox{in }\R^n\times (-1,0),  \\
    w(x,-1)=0     &\mbox{in }\R^n.
    \end{array}
    \quad
    \right.
    \]
    Note that we can extend $w\equiv0$ in $\R^n\times(-\infty,-1)$, which is still a solution of $w_t-\P^-(D^2w)=-f^-\chi_{Q_{1}}$ in $\R^n\times(-\infty,\textcolor{black}{0})$.
    \textcolor{black}{
    By the comparison principle, it follows that
    \begin{equation}\label{e164}
    -\int_{-\infty}^t\|f(\cdot,s)\chi_{(-1,0)}(s)\|_{L_\infty(B_1)}ds\leq w(x,t)\leq0\quad\mbox{in }\R^n\times (-\infty,0).
    \end{equation}
    We also have
    \begin{equation}\label{e163}
        \inf_{Q_{1}}v-\inf_{\partial_pQ_{1}}v\geq \inf_{Q_{1}}w.
    \end{equation}
    Indeed, by choosing $\bar{w}(x,t)=-\int_{-\infty}^t\|f(\cdot,s)\chi_{(-1,0)}(s)\|_{L_\infty(B_1)}ds$, which satisfies
    \[
    \bar{w}_t-\P^-(D^2\bar{w})=-\|f(\cdot,t)\chi_{(-1,0)}(t)\|_{L_\infty(B_1)}\leq-f^-\chi_{Q_{1}} \quad\mbox{in }\R^n\times (-\infty,0),
    \]
    and by the comparison principle, the first inequality of \eqref{e164} follows.
    Similarly, by choosing $\hat{w}=0$ satisfying
    \[
    \hat{w}_t-\P^-(D^2\hat{w})=0\geq-f^-\chi_{Q_{1}} \quad\mbox{in }\R^n\times (-\infty,0),
    \]
    the second inequality of \eqref{e164} follows.
    }
    By applying \eqref{e162} to $\tilde{w}(x,t):=R^{-2}w(Rx,R^2t)$, we have
    \begin{align*}
    \|\tilde{w}_t\|_{L_p(Q_{1/2})}+\|D^2\tilde{w}\|_{L_p(Q_{1/2})}&=\|w_t\|_{L_p(Q_R)}
    +\|D^2w\|_{L_p(Q_R)} \\
    &\leq C_6(\|f^-\chi_{Q_1}\|_{L_p(Q_R)}+R^{-2+(n+2)/p}\|w\|_{L_\infty(Q_R)}).
    \end{align*}
    By letting $R\to\infty$, it follows that
    \begin{align*}
    \|w_t\|_{L_p(\R^n\times (-\infty,0))}+\|D^2w\|_{L_p(\R^n\times (-\infty,0))}
    &\leq C_1\|f^-\chi_{Q_1}\|_{L_p(\R^n\times (-\infty,0))}\\
    &\leq C_1\|f^-\|_{L_p(Q_1)}.
    \end{align*}
    Hence the embedding inequality gives
    \[
    \|w\|_{L_\infty(\R^n\times (-1,0))}{\leq C\|w\|_{W^{2,1}_p(\R^n\times (-1,0))}}
    \leq C\|w_t-\Delta w\|_{L_p(\R^n\times (-1,0))}\leq C\|f^-\|_{L_p(Q_1)}.
    \]
    Together with \eqref{e163}, we arrive at \eqref{e161}.
\end{proof}
In our next step, we will establish the elliptic counterpart of the first assertion in Theorem \ref{t4}.
However, it remains unclear whether we can achieve a comparable result for the second assertion \eqref{e471} in the elliptic case.
This disparity arises from the inherent property that solutions to the elliptic Pucci equations in $\R^n$ are generally unbounded, rendering it infeasible to obtain a similar estimate as shown in \eqref{e164} for the elliptic case.

We denote by $q_M$ the critical exponent of generalized maximum principles for fully nonlinear elliptic equations:
\[
q_M:=\inf\left\{p> \max\left\{\frac{n}{2},p_n\right\}:\sup_{u\in W^{2}_{p}(B_1)}\left\{\frac{\|u\|_{L_\infty(B_1)}-\|u\|_{L_\infty(\partial B_1)}}{\|\P^-(D^2u)\|_{L_{p}(B_1)}}\right\}<\infty\right\}.
\]
We also denote by $q_W$ the critical exponent of $W^{2,1}_p$ estimates:
\[
q_W:=\inf\left\{p> \max\left\{\frac{n}{2},p_n\right\}:\sup_{u\in W^{2}_{p}(B_1)}\left\{\frac{\|D^2u\|_{L_{p}(B_{1/2})}}{\|\P^-(D^2u)\|_{L_{p}(B_1)}+\|u\|_{L_{p}(B_1)}}\right\}<\infty\right\}.
\]
Note that $q_M,q_W\leq n-\varepsilon_2$ for some $\varepsilon_2=\varepsilon_2(n,\lambda,\Lambda)>0$ from \cite{E,FS}.
\begin{thm}It follows that $q_M=q_W$.
\end{thm}
\begin{proof}
    The main steps of the proof are similar to those of Theorem \ref{t4}:
    \begin{itemize}
        \item $q_W\leq q_M$ immediately follows by Escauriaza \cite{E}.
        \item To show $q_M\leq q_W$, it suffice to prove that for $p>q_W$ such that $W^2_p$ estimate holds, there is a constant $C_8>0$ such that for $u\in C^2(\overline{B_1})$,
        \begin{equation}\label{e271}
        \|u\|_{L_\infty(B_1)}-\|u\|_{L_\infty(\partial B_1)}\leq C_8\|\P^-(D^2u)\|_{L_p(B_1)}.
        \end{equation}
    \end{itemize}
    We may suppose that $p\leq n-\varepsilon_2$ since the generalized maximum principle is already known for $p> n-\varepsilon_2$.
    Fix $u\in C^2(\overline{B_1})$ and let $f:=\P^-(D^2u)\in C(\overline{B_1})$.
    Then, we have
    \[
    \sup_{B_1}u-\sup_{\partial B_1}u\leq C\|f\|_{L_p(B_1 )}.
    \]
    from $-\Delta u\leq -\P^-(D^2u)=f$ in $B_1$.
    Hence it remains to prove
    \begin{equation}\label{e272}
    -\inf_{B_1}u+\inf_{\partial B_1}u\leq C\|f\|_{L_p(B_1)}.
    \end{equation}
    To obtain \eqref{e272}, we will estimate $u-\inf_{\partial B_1}u$ from below by using a strong solution $w^R\in W^2_{q,\text{loc}}(B_R)$, $q>n-\varepsilon_2$, of
    \[
    \left\{
    \begin{array}{cc}
    -\P^-(D^2w^R)=-f^-\chi_{B_1}     &\mbox{in }B_R,  \\
    w^R=0     &\mbox{in }\partial B_R
    \end{array}
    \quad
    \right.
    \]
    for $R\ge 1$.
    By the comparison principle, $w^R$ is non-positive in $B_R$ and
    \begin{equation}\label{e273}
        \inf_{B_1}u-\inf_{\partial B_1}u\geq \inf_{B_1}w^R-\sup_{\partial B_1}w^R.
    \end{equation}
    Next we aim to show a lower bound of $w^R$.
    Let $\theta \in (\max\{0,1-\lambda(n-1)/\Lambda\}, 2-n/p)$ and let $A>0$ be such that
    \begin{equation}\label{Af}
    -A\theta (\lambda (n-1)-\Lambda(1-\theta))|x|^{-2+\theta}\leq -f^-(x)\chi_{B_1}(x)\quad\mbox{in }\R^n.
    \end{equation}
    Set
    \[
    z(x):=A(|x|^\theta-R^\theta).
    \]
    \textcolor{black}{
    }
    Since $z$ is $C^2$ except at $0$ and $z-\phi$ does not attains its maximum at $0$ for \textcolor{black}{any $\phi\in W^2_q(B_R)$ and $q>n-\varepsilon_2$, together with \eqref{Af}, $z$ is an $L_q$-viscosity subsolution} of
    \[
    -\P^+(D^2z)\leq \textcolor{black}{-f^-(x)\chi_{B_1}(x)}\quad\mbox{in }B_R.
    \]
    By applying the comparison principle to $w^R$ and $z$ (see Theorem 2.10 in \cite{CCKS} for comparison principles), we have $z\leq w^R$ in $B_R$ and hence
    \begin{equation}\label{e274}
        \|w^R\|_{L_\infty(B_R)}\leq AR^{\theta}.
    \end{equation}
    The $W^2_p$ estimate gives
    \begin{align*}
    \|D^2w^R\|_{L_p(B_{R/2})}
    \leq C(\|f^-\|_{L_p(B_1)}+R^{-2+n/p}\|w^R\|_{L_\infty(B_R)})
    \end{align*}
    and together with the Sobolev inequality,
    \begin{align*}
    \|Dw^R\|_{L_{p^*}(B_{R/2})}&\leq C(\|D^2w^R\|_{L_{p}(B_{R/2})}+R^{-2}\|w^R\|_{L_p(B_R)})\\
    &\leq C(\|f^-\|_{L_p(B_1)}+R^{-2+n/p}\|w^R\|_{L_\infty(B_R)}),
    \end{align*}
    where $p^*=pn/(n-p)>n$.
    Moreover, by Morrey's lemma we have
    \[
        \mathrm{osc}_{B_1} w^R\leq C\|Dw^R\|_{L_{p^*}(B_1)}\leq C(\|f^-\|_{L_p(B_1)}+R^{-2+n/p}\|w^R\|_{L_\infty(B_R)}).
    \]
    From \eqref{e273} we obtain
    \[
        -\inf_{B_1}u+\inf_{\partial B_1}u\leq C(\|f^-\|_{L_p(B_1)}+R^{-2+n/p}\|w^R\|_{L_\infty(B_R)}).
    \]
    Noting \eqref{e274} and $\theta<2-n/p$, we send $R\to\infty$ to derive
    \begin{align*}
        -\inf_{B_1}u+\inf_{\partial B_1}u\leq C\|f^-\|_{L_p(B_1)}.
    \end{align*}
    The theorem is proved.
\end{proof}
\subsection{Global \texorpdfstring{$W^2_p$}{1} estimates}
We finish this section {by} providing certain global $W^2_p$ estimates.
Let us consider the following elliptic equation:
\begin{equation}\label{t43}
    \tau u-\P^-(D^2u)=f\quad\mbox{in }\T^n,
\end{equation}
where $\tau>0$, $f\in C(\T^n)$ and $\T^n$ is the flat torus $\R^n/\ZZ^n$ (i.e., $x\in \R^n$ is identified with $x+y$ for $y\in\ZZ^n$).
It is well known that there exists a unique strong solution $u\in W^{2}_q(\T^n)$ for $q\geq1$ of \eqref{t43}.
An important feature of our work is that we obtain uniform estimates in $\tau\in(0,1)$, which is important in the study of ergodic problems.
In what follow, we write
\begin{equation}\label{e431}
    c_\tau:=\tau\int_{\T^n}udx\quad\mbox{and}\quad v:=u-\int_{\T^n}udx .
\end{equation}
For $p\in[1,\infty)$, we say that the interior $W^2_p$ estimate holds if there exists $C>0$ such that for any $\phi\in W^2_p(B_1)$, it follows that
\[
\|D^2\phi\|_{L_p(B_{1/2})}\leq C\left(\|\P^-(D^2\phi)\|_{L_p(B_1)}+\|\phi\|_{L_p(B_1)}\right).
\]
We will use mollifications of periodic functions, defined as follows:
\begin{dfn}
    We fix a mollifier $\eta\in C_0^\infty(\R^n)$ satisfying $\|\eta\|_{L_1(\R^n)}=1$ and $\eta\geq0$ in $\R^n$.
    $u\in C(\T^n)$ may be extended as a function in $\R^n$, defined by $u(x+y)=u(x)$ for $x\in\T^n$ and $y\in\ZZ^n$.
    The mollification $u^\delta$ of $u$ is defined by
    \[
    u^\delta(x) :=\int_{\R^n}\eta^\delta\left(x-y\right) u(y)dy,
    \]
    where $\eta^\delta(x):=\delta^{-n}\eta(\delta^{-1}x)$.
\end{dfn}
Note that the mollification $u^\delta$ is also periodic (i.e., $u^\delta(x+y)=u^\delta(x) $ for $x\in\T^n$ and $y\in\ZZ^n$).
We also have analogous definitions of mollifications of measurable functions and Radon measures.
\begin{lem}\label{lem48}
    Let $p\geq1$ be such that the $W^{2}_p$ estimate holds.
    Then, there exists a constant $C_9>0$ such that for any $\tau\in(0,1)$ and $f\in C(\T^n)$, it follows that
    \[
        \|v\|_{L_p(\T^n)}+\|D^2u\|_{L_p(\T^n)}\leq C_9\left(\|f\|_{L_p(\T^n)}+|c_\tau|\right).
    \]
\end{lem}
\begin{proof}
    Note that $v$ is a strong solution of $\tau v-\P^-(D^2v)=f-c_\tau$ in $\T^n$ and has zero average.
    By using the interior $W^2_p$ estimate and a standard covering argument, we have
    \begin{equation}\label{e482}
        \|D^2u\|_{L_p(\T^n)}=\|D^2v\|_{L_p(\T^n)}\leq C\left(\|f\|_{L_p(\T^n)}+|c_\tau|+\|v\|_{L_p(\T^n)}\right).
    \end{equation}
    Hence it remains to show
    \[
        \|v\|_{L_p(\T^n)}\leq C\left(\|f\|_{L_p(\T^n)}+|c_\tau|\right).
    \]
    {We prove by contradiction. }Suppose that for any $N\in\N$, there exist $\tau_N\in (0,1)$ and $f_N\in C(\T^n)$ such that the strong solution $u_N\in W^2_p(\T^n)$ of
    \begin{equation}\label{e481}
        \tau_N u_N-\P^-(D^2u_N)=f_N\quad\mbox{in }\T^n,
    \end{equation}
    satisfies
    \[
        \|v_N\|_{L_p(\T^n)}\geq N\left(\|f_N\|_{L_p(\T^n)}+|c_N|\right),
    \]
    where $c_N=\tau_N\int_{\T^n} u_N$ and $v_N=u_N-\int_{\T^n} u_N$.
    We shall show that this leads to a contradiction.

    Dividing $v_N$ by $\|v_N\|_{L_p(\T^n)}$, it {suffices} to consider $\|v_N\|_{L_p(\T^n)}=1$ and $\|f_N\|_{L_p(\T^n)},|c_N|\to0$ as $N\to\infty$.
    Then, from \eqref{e482} and the Kondrachov compactness theorem, we obtain $\tau_{\infty}\in [0,1]$ and $v_{\infty}\in W^1_p$ such that, up to a subsequence, $\tau_N\to\tau_{\infty}$ and $v_N\to v_{\infty}$ strongly in $W^1_p$ as $N\to\infty$.
    We have $\|v_{\infty}\|_{L_p(\T^n)}=\lim_{N\to\infty}\|v_N\|_{L_p(\T^n)}=1$.

    Let $v^\delta_N$ and $v^\delta_\infty$ be the mollifications of $v_N$ and $v_\infty$ respectively.
    Note that $v_N^\delta$ and $v^\delta_\infty$ also have zero average.
    Then, from {the} concavity of $\P^-$, we have
    \begin{align*}
        \tau_N v_N^\delta(x)-\P^-(D^2v_N^\delta(x))
        &\leq \int_{\R^n}\eta^\delta\left(x-y\right)(\tau_N v_N-\P^-(D^2v_N))dy\\
        &\leq f_N^\delta(x)-c_N\quad\mbox{in }\T^n,
    \end{align*}
    where $f_N^\delta$ is the mollification of $f_N$.
    Hence sending $N\to\infty$, we have
    \[
    \tau_\infty v^\delta_\infty-\P^-(D^2v^\delta_\infty)\leq 0\quad\mbox{in }\T^n.
    \]
    If $\tau_\infty>0$, we apply the maximum principle to see $v^\delta_\infty\leq 0$ in $\T^n$, and because $\int_{\T^n}v^\delta_\infty=0$, we get $v^\delta_\infty=0$ in $\T^n$.
    If $\tau_\infty=0$, by the strong maximum principle (see \cite{GT} for instance), $v^\delta_\infty$ is constant in $\T^n$ and because $\int_{\T^n}v^\delta_\infty=0$, $v^\delta_\infty=0$ in $\T^n$.
    In any case, we have $v^\delta_\infty=0$ in $\T^n$ and by sending $\delta\to0$, $v_\infty=0$ a.e. in $\T^n$; a contradiction in virtue of $\|v_\infty\|_{L_p(\T^n)}=1$.
\end{proof}
According to Lemma \ref{lem48}, it only remains to estimate $c_\varepsilon$.
To this end, we need some additional conditions.
First, we suppose that $f$ is nonnegative and show the global $W^2_1$ estimate.
\begin{thm}\label{thm49}
    Suppose that $f\in C(\T^n)$ is nonnegative.
    Let $u \in W^2_1(\T^n)$ be the strong solution of \eqref{t43}.
    Then for any $\tau>0$, it follows that
    \[
    \tau\|u\|_{L_1(\T^n)}+\frac{\Lambda-\lambda}{2}\|D^2u\|_{L_1(\T^n)}= \|f\|_{L_1(\T^n)}.
    \]
\end{thm}
\begin{proof}
    Since $f$ is nonnegative, the maximum principle implies $u$ is also nonnegative.
    Then, {by \eqref{eq3.44},}
    \begin{align*}
        \int_{\T^n}fdx
        =\int_{\T^n}({\tau} u-\P^-(D^2u))dx
        =\int_{\T^n}\left({\tau} u+\frac{\Lambda-\lambda}{2}|D^2u|\right)dx.
    \end{align*}
\end{proof}
We next suppose that the generalized maximum principle holds {for some} $p>q_M=q_W$.
This corresponds to the condition (P) in \cite{CCKS} and hence we may use the stability property of $L_p$-viscosity solutions (see Theorem 3.8 in \cite{CCKS}).
\begin{thm}\label{thm410}
    Suppose that $p>q_M=q_W$.
    Then, there exists a constant $C_{10}>0$ such that for any $\tau\in (0,1)$ and any $f\in C(\T^n)$, it follows that
    \[
    |c_\tau|\leq C_{10}\|f\|_{L_p(\T^n)}.
    \]
\end{thm}
\begin{proof}
    Let us suppose by contradiction that for any $N\in\N$, there exist $\tau_N\in (0,1)$ and $f_N\in C(\T^n)$ such that
    \[
    |c_N|\geq N\|f_N\|_{L_p(\T^n)},
    \]
    where $c_N=\tau_N\int_{\T^n}u_N$ and $u_N\in W^2_p(\T^n)$ is the strong solution of
    \[
    \tau_N u_N-\P^-(D^2u_N)=f_N\quad\mbox{in }\T^n.
    \]
    Dividing $u_N$ by $|c_N|$, we may suppose $|c_N|=1$ and $\|f_N\|_{L_p(\T^n)}\to0$.
    By using Lemma~\ref{lem48}, we have
    \[
        \|v_N\|_{L_p(\T^n)}+\|D^2u_N\|_{L_p(\T^n)}\leq C\left(\|f_N\|_{L_p(\T^n)}+1\right).
    \]
    We now have two cases.

    1. In the case that, up to a subsequence, $\tau_N$ converges to some $\tau_\infty\in (0,1]$, we can extract a subsequence of $u_N$, which uniformly converges to $u_\infty$ in $\T^n$.
    By using the stability of $L_p$-viscosity solutions (see Theorem 3.8 of \cite{CCKS}), $u_\infty$ is an $L_p$-viscosity solution of
    \[
    \tau_\infty u_\infty-\P^-(D^2u_\infty)=0\quad\mbox{in }\T^n.
    \]
    Therefore we have $u\equiv0$ in $\T^n$, which implies $c_N\to 0$; a contradiction in virtue of $|c_N|=1$.

    2. In the case of $\tau_N\to 0$ as $N\to\infty$, we extract subsequences of $v_N$ which converges uniformly to some function $v_\infty\in C(\T^n)$ in $\T^n$ and $c_N$ which converges to some constant $c_\infty\in\{-1,1\}$.
    We again use the stability of $L_p$-viscosity solutions, to see that $v_\infty$ is an $L_p$-viscosity solution of
    \[
    -\P^-(D^2v_\infty)=-c_\infty\quad\mbox{in }\T^n.
    \]
    It is easy to see that $c_\infty=0$.
    Indeed, let us consider the maximum (resp., minimum) point $x_0\in\T^n$ of $v_\infty$.
    Then, from the definition of $L_p$-viscosity solutions, we have $c_\infty\geq0$ (resp., $c_\infty\leq0$).
    This is a contradiction with $c_\infty\in \{-1,1\}$.
\end{proof}
By applying our regularity results, we shall present {the} existence and uniqueness results.
\begin{thm}\label{thm413}
    Let  $f\in L_p(\T^n)$ for {some} $p>q_M=q_W$.
    Then there exists a unique strong solution $u\in W^2_p(\T^n)$ of \eqref{t43}.
    Moreover there exists a constant $C_{11}>0$ depending only on $n,\lambda$ and $\Lambda$ such that
    \[
    |c_\tau|+\|D^2u\|_{L_p(\T^n)}\leq C_{11}\|f\|_{L_p(\T^n)}.
    \]
\end{thm}
\begin{proof}
    We consider the approximating equation:
    \[
    \tau u^\delta-\P^-(D^2u^\delta)=f^\delta\quad\mbox{in }\T^n
    \]
    for the mollification $f^\delta$ of $f$.
    From Lemma \ref{lem48} and Theorem \ref{thm410}, we obtain
    \[
    |c^\delta_\tau|+\|v^\delta\|_{L_p(\T^n)}+\|D^2u\|_{L_p(\T^n)}\leq C\|f\|_{L_p(\T^n)},
    \]
    where $c^\delta_\tau$ and $v^\delta$ are defined {in} the same way as in \eqref{e431}.
    Extracting a subsequence, there exists $u\in W^2_p(\T^n)$ such that $u^\delta$ converges to $u$ uniformly in $\T^n$.
    From the stability property of $L_p$-viscosity solution (see Theorem 3.8 of \cite{CCKS}), $u$ is an $L_p$-viscosity (and strong) solution of \eqref{t43}.
    Moreover, the uniqueness of the solution follows from the comparison principle (see Theorem 2.10 of \cite{CCKS}).
\end{proof}
As we mentioned in the introduction, if the right-hand side $f$ is only $L_1$ or in the space of Radon measures, the notion of $L_p$-viscosity solutions is no longer suitable, since we cannot expect continuous solutions.
In the next theorem, we instead provide the existence of the largest subsolution for the Pucci equation.
Then, by Proposition \ref{prop415}, we will also show the consistency property: if the largest subsolution is $W^2_1$, then it is a strong solution.
\begin{thm}\label{thm414}
    Let $\lambda<\Lambda$ and $\mu$ be a nonnegative Radon measure in $\T^n$.
     Then there exists {a unique} $u\in L_1(\T^n)$ with the following properties:
    \begin{enumerate}
        \item[(i)] For any nonnegative function $\phi\in C^\infty_0(\R^n)$ and any constant symmetric matrix $A$ such that $\lambda \mathrm{I}\leq A\leq\Lambda \mathrm{I}$, we have
    \[
    \int_{\R^n}u(\tau\phi-\Tr AD^2\phi)dx\leq \int_{\R^n}\phi d\mu.
    \]
        \item[(ii)] If $v\in L_1(\T^n)$ satisfies (i), then $v\leq u$ a.e. in $\T^n$.
    \end{enumerate}
    Moreover it holds that
    \begin{equation}\label{e4121}
        \tau\|u\|_{L_1(\T^n)}+\frac{\Lambda-\lambda}{2}\|D^2u\|_{\M(\T^n)}\leq \|\mu\|_{\M(\T^n)}.
    \end{equation}
\end{thm}
\begin{proof}
    Let $f^\delta$ be the mollification of $\mu$, defined by
    \[
    f^\delta(x)=\int_{\R^n}\eta^\delta(x-y) d\mu(y).
    \]
    Then there exists a classical solution $u^\delta\in C^2(\T^n)$ of
    \[
    \tau u^\delta-\P^-(D^2u^\delta)=f^\delta\quad\mbox{in }\T^n.
    \]
    From Theorem \ref{thm49}, we also have
    \begin{equation*}
        \tau\|u^\delta\|_{L_1(\T^n)}+\frac{\Lambda-\lambda}{2}\|D^2u^\delta\|_{L_1(\T^n)}= \|f^\delta\|_{L_1(\T^n)}= \|\mu\|_{\M(\T^n)}.
    \end{equation*}
    From the Kondrachov compactness theorem, there exists $u\in L_1(\T^n)$ such that, up to subsequence, $u^\delta$ converge to $u$ in $L_1(\T^n)$.
    Moreover, by the weak-$\ast$ compactness in $\M(\T^n)$, we also have $D^2u^\delta\overset{\ast}{\rightharpoonup}D^2u$ in $\M(\T^n)$ and
    \[
    \|D^2u\|_{\M(\T^n)}\leq \liminf_{\delta\to0}\|D^2u^\delta\|_{\M(\T^n)}=\liminf_{\delta\to0}\|D^2u^\delta\|_{L_1(\T^n)}.
    \]
    Thus we obtain \eqref{e4121}.
    Since
    \begin{equation*}
    \P^-(X)=\inf\left\{\Tr AX:\lambda\mathrm{I}\leq A\leq \Lambda\mathrm{I}\right\}
    \end{equation*}
    for any symmetric matrix $X$, $u^\delta$ satisfies
    \begin{equation*}
    \tau u^\delta-\Tr AD^2u^\delta\leq f^\delta\quad\mbox{in }\T^n.
    \end{equation*}
    Testing nonnegative $\phi\in C^\infty_0(\R^n)$, we have
    \begin{equation}\label{e4111}
        \int_{\R^n}u^\delta(\tau\phi-\Tr AD^2\phi)dx\leq \int_{\R^n}\phi f^\delta dx=\int_{\R^n}\phi^\delta d\mu.
    \end{equation}
    By letting $\delta\to0$ in \eqref{e4111}, (i) follows.

    To prove (ii), fix $v\in L_1(\T^n)$ satisfying (i).
    Let $v^\delta$ be the mollification of $v$.
    Then from (i), we have
    \begin{align*}
        \tau v^\delta(x)-\Tr AD^2v^\delta(x)
        &=\int_{\R^n}v(y)(\tau\eta^\delta(x-y)-\Tr AD^2\eta^\delta(x-y))dx\\
        &\leq f^\delta(x) \quad\mbox{in }\T^n
    \end{align*}
    and by taking the maximum of $A$ on the left-hand side, $\tau v^\delta-\P^-(D^2v^\delta)\leq f^\delta$ in $\T^n$.
    Hence the comparison principle gives $v^\delta\leq u^\delta$ in $\T^n$.
    We conclude the result by sending $\delta\to0$.
\end{proof}
\begin{prop}\label{prop415}
    Let $f\in L_1(\T^n)$ be nonnegative.
    Let $u$ be as in Theorem \ref{thm414} with the Radon measure $\mu$, defined by $\mu(B)=\int_{B}fdx$ for each Borel set $B\subset\R^n$.
    If $u\in W^2_1(\T^n)$, then it follows that
    \begin{equation}\label{e4151}
        \tau u-\P^-(D^2u)=f\quad\mbox{a.e. in }\T^n.
    \end{equation}
\end{prop}
\begin{proof}
    From (i) of Theorem \ref{thm414} and by taking the maximum of $A$, it follows that
    \[
    \tau u-\P^-(D^2u)\leq f\quad\mbox{a.e. in }\T^n.
    \]
    To prove \eqref{e4151}, we suppose, on the contrary, that there exist a measurable set $B\subset \T^n$ and $\theta>0$ such that $|B|>0$ and
    \[
    \tau u-\P^-(D^2u)\leq f-\theta\chi_B \quad\mbox{a.e. in }\T^n.
    \]
    From Theorem \ref{thm413}, there exists a unique $L_p$-strong solution $w\in W^2_p(\T^n)$ of
    \[
    \tau w-\P^-(D^2w)= \theta\chi_B \quad\mbox{a.e. in }\T^n
    \]
    for $p>q_M$.
    From the comparison principle, it follows that $w\geq0$ and $w\not\equiv0$ in $\T^n$.
    We compute
    \[
     \tau (u+w)-\P^-(D^2(u+w))\leq f \quad\mbox{a.e. in }\T^n.
    \]
    Hence $u+w$ satisfies (i), which together with (ii), implies $u+w\leq u$ in $\T^n$.
    Therefore $w\leq0$ in $\T^n$; a contradiction to $w\geq0$ and $w\not\equiv0$ in $\T^n$.
\end{proof}
\begin{rem}
    The estimate
    \[
    \tau\|u\|_{L_p(\R^n)}\leq C\|f\|_{L_p(\R^n)}
    \]
    of the solution $u$ to \eqref{t43} with $\T^n$ replaced by $\R^n$, has already been obtained for any $\tau>0$ and $p\geq n$ in \cite{Kbook}.
    See also \cite{DKL,K10} for the $W^2_p$ estimates of elliptic equations in $\R^n$.
    On the other hand, our method allows to extend the range of the exponent $p$ although we need the compactness of the domain to use the Kondrachov theorem.
    Our approach in Theorem \ref{thm410} also heavily relies on the stability properties of $L_p$-viscosity solutions.
    It is still open whether there exists any appropriate notion of solutions for equations with the $L_1$-right hand side $f$ (even when $f\geq0$), which possesses stability properties.
\end{rem}

\noindent {\bf Data availability}
Data sharing not applicable to this article and no data sets were generated or analyzed during the current study.
\bibliographystyle{plain}
\bibliography{sample}
\end{document}